\documentclass[preprint,12pt]{elsarticle}
\usepackage{amsmath,amssymb,amsfonts,amsbsy}
\usepackage{tikz-cd}
\usepackage{amscd}
\usepackage{wasysym}
\usetikzlibrary{matrix,arrows,decorations.pathmorphing}
\usepackage{graphicx}
\usepackage{float}
\usepackage{setspace}
\usepackage[margin=1.0in]{geometry}

\newtheorem{thm}{Theorem}[section]

\newtheorem{cor}[thm]{Corollary}
\newtheorem{prop}[thm]{Proposition}
\newproof{pf}{Proof}

\newtheorem{remark}[thm]{Remark}

\newcommand{\mc}[1]{\mathcal{#1}}
\newcommand{\e}[1]{\emph{#1}}

\newcommand{\la}{\langle}
\newcommand{\ra}{\rangle}

\newcommand{\rmv}[1]{}

\newcommand{\LG}{VN(G)}

\newcommand{\LO}{L^1(G)}

\newcommand{\LT}{L^2(G)}

\newcommand{\LI}{L^{\infty}(G)}

\newcommand{\BH}{\mc{B}(H)}

\newcommand{\BLT}{\mc{B}(L^2(G))}

\newcommand{\TC}{\mc{T}(L^2(G))}

\newcommand{\vphi}{\varphi}

\newcommand{\lm}{\lambda}

\newcommand{\Gam}{\Gamma}
\newcommand{\om}{\omega}

\newcommand{\ten}{\otimes}
\newcommand{\oten}{\overline{\otimes}}
\newcommand{\hten}{\widehat{\otimes}}

\newcommand{\id}{\textnormal{id}}
\newcommand{\h}[1]{\widehat{#1}}

\providecommand{\norm}[1]{\lVert#1\rVert}

\newcommand{\F}{\mathbb{F}}

\newcommand{\C}{\mathbb{C}}

\newcommand{\R}{\mathbb{R}}

\journal{Journal of Functional Analysis}

\begin{document}

\begin{frontmatter}



\title{On the operator homology of the Fourier algebra and its $cb$-multiplier completion}
\author[JC1]{Jason Crann}\ead{jcrann@math.carleton.ca}
\author[ZT]{Zsolt Tanko}\ead{ztanko@uwaterloo.ca}

\address[JC1]{School of Mathematics and Statistics, Carleton University, Ottawa, ON, Canada H1S 5B6}
\address[ZT]{Department of Pure Mathematics, University of Waterloo, Waterloo, ON, Canada N2L 3G1}


\begin{abstract}
We study various operator homological properties of the Fourier algebra $A(G)$ of a locally compact group $G$. Establishing the converse of two results of Ruan and Xu \cite{RX}, we show that $A(G)$ is relatively operator 1-projective if and only if $G$ is IN, and that $A(G)$ is relatively operator 1-flat if and only if $G$ is inner amenable. We also exhibit the first known class of groups for which $A(G)$ is not relatively operator $C$-flat for any $C\geq1$. As applications of our techniques, we establish a hereditary property of inner amenability, answer an open question of Lau and Paterson \cite{LP1}, and answer an open question of Anantharaman--Delaroche \cite{Claire} on the equivalence of inner amenability and Property (W). In the bimodule setting, we show that relative operator 1-biflatness of $A(G)$ is equivalent to the existence of a contractive approximate indicator for the diagonal $G_\Delta$ in the Fourier--Stieltjes algebra $B(G\times G)$, thereby establishing the converse to a result of Aristov, Runde, and Spronk \cite{ARS}. We conjecture that relative $1$-biflatness of $A(G)$ is equivalent to the existence of a quasi-central bounded approximate identity in $\LO$, that is, $G$ is QSIN, and verify the conjecture in many special cases. We finish with an application to the operator homology of $A_{cb}(G)$, giving examples of weakly amenable groups for which $A_{cb}(G)$ is not operator amenable.
\end{abstract}


\end{frontmatter}


\section{Introduction}

The operator homology\let\thefootnote\relax\footnotetext{2010 \e{Mathematics Subject Classification:} Primary 46L07 46H25 46M10, Secondary 46M18 43A15.} of the Fourier algebra $A(G)$ of a locally compact group $G$ has been a topic of interest in abstract harmonic analysis since Ruan's seminal work \cite{Ru3}, where, among other things, he established the equivalence of amenability of $G$ and operator amenability of $A(G)$. From the perspective of Pontryagin duality, this result is the dual analogue of Johnson's celebrated equivalence of amenability of $G$ and (operator) amenability of $\LO$ \cite{John}. In much the same spirit, dual analogues of various homological properties of $\LO$ were established within the category of operator $A(G)$-modules, including the operator weak amenability of $A(G)$ \cite{Sp}, and the equivalence of discreteness of $G$ and relative operator biprojectivity of $A(G)$ \cite{Ar2,Wo}.

Continuing in this spirit, Ruan and Xu (implicity) showed that $A(G)$ is relatively operator 1-projective whenever $G$ is an IN group (see also \cite{FLS}), and that $A(G)$ is relatively operator 1-flat whenever $G$ is inner amenable \cite{RX}. In this paper, we establish the converse of both of these results, and exhibit the first known class of groups -- including every connected non-amenable group -- for which $A(G)$ is not relatively operator $C$-flat for any $C\geq1$. Along the way, we show that inner amenability passes to closed subgroups, answer an open question of Lau and Paterson \cite{LP1}, and answer an open question of Anantharaman--Delaroche \cite[Problem 9.1]{Claire} on the equivalence of inner amenability and Property (W).

The relative operator biflatness of $A(G)$ has been studied by Ruan and Xu \cite{RX} and Aristov, Runde, and Spronk \cite{ARS}, where it was shown (by different methods) that $A(G)$ is relatively operator biflat whenever $G$ is QSIN, meaning $\LO$ has a quasi-central bounded approximate identity (see \cite{ARS,LR,St2}). The approach of Aristov, Runde, and Spronk is via approximate indicators, where they show that $A(G)$ is relatively operator $C$-biflat whenever the diagonal subgroup $G_\Delta\leq G\times G$ has a bounded approximate indicator in $B(G\times G)$ of norm at most $C$. One of the main results of this paper establishes the converse when $C=1$, that is, $A(G)$ is relatively operator 1-biflat if and only if $G_\Delta$ has a contractive approximate indicator in $B(G\times G)$. Recalling that $A(G)$ is operator amenable precisely when $A(G\times G)$ has a bounded approximate diagonal \cite{Ru3}, we see that $A(G)$ is relatively operator 1-biflat precisely when $A(G\times G)$ has a contractive approximate diagonal in the \e{Fourier--Stieltjes algebra} $B(G\times G)$, a result which elucidates the relationship between operator amenability and relative operator biflatness for $A(G)$, and for completely contractive Banach algebras more generally.

We conjecture that relative operator $1$-biflatness of $A(G)$ is equivalent to the QSIN condition, and we verify the conjecture in many special cases. For a discrete group $H$ acting ergodically by automorphisms on a compact group $K$, we also establish a connection between relative operator biflatness of $A(K\rtimes H)$ and the existence of $H$-invariant means on $L^{\infty}(K)$ distinct from the Haar integral.

Combining results of Leptin \cite{Lep} and Ruan \cite{Ru3}, we see that $A(G)$ has a bounded approximate identity precisely when it is operator amenable. It is known that $G$ is weakly amenable if and only if the algebra $A_{cb}(G)$ has a bounded approximate identity \cite{Fo}, and it was suggested in \cite{FRS} that $A_{cb}(G)$ may be operator amenable exactly when $G$ is weakly amenable. We finish the paper by providing a large family of counter-examples, which includes every weakly amenable, non-amenable, almost connected group.


\section{Preliminaries}

Let $\mc{A}$ be a completely contractive Banach algebra. We say that an operator space $X$ is a right \e{operator $\mc{A}$-module} if it is a right Banach $\mc{A}$-module such that the module map $m_X:X\hten\mc{A}\rightarrow X$ is completely contractive, where $\hten$ denotes the operator space projective tensor product. We say that $X$ is \e{faithful} if for every non-zero $x\in X$, there is $a\in\mc{A}$ such that $x\cdot a\neq 0$, and we say that $X$ is \e{essential} if $\la X\cdot\mc{A}\ra=X$, where $\la\cdot\ra$ denotes the closed linear span. We denote by $\mathbf{mod}-\mc{A}$ the category of right operator $\mc{A}$-modules with morphisms given by completely bounded module homomorphisms. Left operator $\mc{A}$-modules and operator $\mc{A}$-bimodules are defined similarly, and we denote the respective categories by $\mc{A}-\mathbf{mod}$ and $\mc{A}-\mathbf{mod}-\mc{A}$.


\begin{remark} Regarding terminology, in what follows we will often omit the term ``operator'' when discussing homological properties of operator modules as we will be working exclusively in the operator space category.
\end{remark}

Let $\mc{A}$ be a completely contractive Banach algebra, $X$ in $\mathbf{mod}-\mc{A}$ and $Y$ in $\mc{A}-\mathbf{mod}$. The \e{$\mc{A}$-module tensor product} of $X$ and $Y$ is the quotient space $X\hten_{\mc{A}}Y:=X\hten Y/N$, where
$$N=\la x\cdot a\ten y-x\ten a\cdot y\mid x\in X, \ y\in Y, \ a\in\mc{A}\ra,$$
and, again, $\la\cdot\ra$ denotes the closed linear span. It follows that
$$\mc{CB}_{\mc{A}}(X,Y^*)\cong N^{\perp}\cong(X\hten_{\mc{A}} Y)^*,$$
where $\mc{CB}_{\mc{A}}(X,Y^*)$ denotes the space of completely bounded right $\mc{A}$-module maps $\Phi:X\rightarrow Y^*$.
If $Y=\mc{A}$, then clearly $N\subseteq\mathrm{Ker}(m_X)$ where $m_X:X\hten\mc{A}\rightarrow X$ is the multiplication map. If the induced mapping $\widetilde{m}_X:X\hten_{\mc{A}}\mc{A}\rightarrow X$ is a completely isometric isomorphism we say that $X$ is an \e{induced $\mc{A}$-module}. A similar definition applies for left modules. In particular, we say that $\mc{A}$ is \e{self-induced} if $\widetilde{m}_\mc{A}:\mc{A}\hten_{\mc{A}}\mc{A}\cong\mc{A}$ completely isometrically.

Let $\mc{A}$ be a completely contractive Banach algebra and $X$ in $\mathbf{mod}-\mc{A}$. The identification $\mc{A}^+=\mc{A}\oplus_1\C$ turns the unitization of $\mc{A}$ into a unital completely contractive Banach algebra, and it follows that $X$ becomes a right operator $\mc{A}^+$-module via the extended action
\begin{equation*}x\cdot(a+\lm e)=x\cdot a+\lm x, \ \ \ a\in\mc{A}^+, \ \lm\in\C, \ x\in X.\end{equation*}
Let $C\geq1$. We say that $X$ is \e{relatively $C$-projective} if there exists a morphism $\Phi^+:X\rightarrow X\hten\mc{A}^+$ satisfying $\norm{\Phi^+}_{cb}\leq C$ which is a right inverse to the extended module map $m_X^+:X\hten\mc{A}^+\rightarrow X$. When $X$ is essential, this is equivalent to the existence of a morphism $\Phi:X\rightarrow X\hten\mc{A}$ satisfying $\norm{\Phi}_{cb}\leq C$ and $m_X\circ\Phi=\id_{X}$ by the operator analogue of \cite[Proposition 1.2]{DP}.




Given a completely contractive Banach algebra $\mc{A}$ and $X$ in $\mathbf{mod}-\mc{A}$, there is a canonical completely contractive morphism $\Delta_X^+:X\rightarrow\mc{CB}(\mc{A}^+,X)$ given by
\begin{equation*}\Delta_X^+(x)(a)=x\cdot a, \ \ \ x\in X, \ a\in\mc{A}^+,\end{equation*}
where the right $\mc{A}$-module structure on $\mc{CB}(\mc{A}^+,X)$ is defined by
\begin{equation*}(\Psi\cdot a)(b)=\Psi(ab), \ \ \ a\in\mc{A}, \ \Psi\in\mc{CB}(\mc{A}^+,X), \ b\in\mc{A}^+.\end{equation*}
An analogous construction exists for objects in $\mc{A}-\mathbf{mod}$. For $C\geq 1$, we say that $X$ is \e{relatively $C$-injective} if there exists a morphism $\Phi^+:\mc{CB}(\mc{A}^+,X)\rightarrow X$ such that $\Phi^+\circ\Delta_X^+=\id_{X}$ and $\norm{\Phi^+}_{cb}\leq C$. When $X$ is faithful, this is equivalent to the existence of a morphism $\Phi:\mc{CB}(\mc{A},X)\rightarrow X$ such that $\Phi\circ\Delta_X=\id_{X}$ and $\norm{\Phi}_{cb}\leq C$ by the operator analogue of \cite[Proposition 1.7]{DP}, where $\Delta_X(x)(a):=\Delta_X^+(x)(a)$ for $x\in X$ and $a\in\mc{A}$.

We say that $X$ is \e{$C$-injective} if for every $Y,Z$ in $\mathbf{mod}-\mc{A}$, every completely isometric morphism $\Psi:Y\hookrightarrow Z$, and every morphism $\Phi:Y\rightarrow X$, there exists a morphism $\widetilde{\Phi}:Z\rightarrow X$ such that $\norm{\widetilde{\Phi}}_{cb}\leq C\norm{\Phi}_{cb}$ and $\widetilde{\Phi}\circ\Psi=\Phi$.



For a completely contractive Banach algebra $\mc{A}$, we say that $X$ in $\mathbf{mod}-\mc{A}$ is \e{relatively $C$-flat} (respectively, \e{$C$-flat}) if its dual $X^*$ is relatively $C$-injective (respectively, $C$-injective) in $\mc{A}-\mathbf{mod}$ with respect to the canonical module structure given by
$$\la a\cdot f,x\ra = \la f, x\cdot a\ra, \ \ \ f\in X^*, \ x\in X, \ a\in\mc{A}.$$
Similar definitions apply to left operator $\mc{A}$-modules. In the case of operator bimodules, we say that $X$ in $\mc{A}-\mathbf{mod}-\mc{A}$ is \e{relatively $C$-biflat} (respectively, \e{$C$-biflat}) if its dual $X^*$ is relatively $C$-injective (respectively, $C$-injective) in $\mc{A}-\mathbf{mod}-\mc{A}$. Viewing $\mc{A}\hten\mc{A}$ as an operator $\mc{A}$-bimodule via
$$a\cdot(b\ten c)=ab\ten c, \ \ (b\ten c)\cdot a=b\ten ca, \ \ \ a,b,c\in\mc{A},$$
we say that $\mc{A}$ is \e{operator amenable} if it is relatively $C$-biflat in $\mc{A}-\mathbf{mod}-\mc{A}$ for some $C\geq1$, and has a bounded approximate identity. By \cite[Proposition 2.4]{Ru3} this is equivalent to the existence of a bounded approximate diagonal in $\mc{A}\hten\mc{A}$, that is, a bounded net $(A_\alpha)$ in $\mc{A}\hten\mc{A}$ satisfying
$$a\cdot A_\alpha - A_\alpha\cdot a\rightarrow 0, \  \ m_{\mc{A}}(A_\alpha)\cdot a \rightarrow a, \ \ \ a\in\mc{A}.$$


For a locally compact group $G$, the left and right regular representations $\lm,\rho:G\rightarrow\BLT$ are given by
$$\lm(s)\xi(t)=\xi(s^{-1}t), \ \ \rho(s)\xi(t)=\xi(ts)\Delta(s)^{1/2}, \ \ \ s,t\in G, \ \xi\in\LT.$$
The von Neumann algebra generated by $\lm(G)$ is called the \e{group von Neumann algebra} of $G$ and is denoted by $VN(G)$. It is known that $\LG$ is \e{standardly represented} on $\LT$ (cf. \cite{H}), so that every normal state $\om\in \LG_*$ is the restriction of a vector state $\om_\xi$ to $\LG$ for a unique unit vector $\xi\in\mc{P}:=\overline{\{\eta\ast J\eta\mid \eta\in C_c(G)\}}$ \cite[Lemma 2.10]{H}, where $C_c(G)$ denotes the continuous functions on $G$ with compact support, and $J$ is the conjugate linear isometry given by
$$J\eta(s)=\overline{\eta(s^{-1})}\Delta(s^{-1})^{1/2}, \ \ \ s\in G, \ \eta\in\LT.$$

The set of coefficient functions of the left regular representation,
\begin{equation*}A(G)=\{u:G\rightarrow\C : u (s)=\la\lm(s)\xi,\eta\ra, \ \xi,\eta\in\LT, \ s\in G\},\end{equation*}
is called the \e{Fourier algebra} of $G$. It was shown by Eymard that, endowed with the norm
$$\norm{u}_{A(G)}=\text{inf}\{\norm{\xi}_{\LT}\norm{\eta}_{\LT} : u(\cdot)=\la\lm(\cdot)\xi,\eta\ra\},$$
$A(G)$ is a Banach algebra under pointwise multiplication \cite[Proposition 3.4]{E}. Furthermore, it is the predual of $\LG$, where the duality is given by
\begin{equation*}\la u,\lm(s)\ra=u(s),\ \ \ u\in A(G), \ s\in G.\end{equation*}
Eymard also showed that the space of functions $\vphi:G\rightarrow\C$ for which there exists a strongly continuous unitary representation $\pi:G\rightarrow\mc{B}(H_\pi)$ and $\xi,\eta\in H_\pi$ such that $\vphi(s)=\la\pi(s)\xi,\eta\ra$, $s\in G$, is a unital Banach algebra (with pointwise multiplication) under the norm
$$\norm{\vphi}_{B(G)}=\text{inf}\{\norm{\xi}_{H_\pi}\norm{\eta}_{H_\pi} : \vphi(\cdot)=\la\pi(\cdot)\xi,\eta\ra\},$$
called the  \e{Fourier-Stieltjes algebra} of $G$ \cite[Proposition 2.16]{E}, denoted by $B(G)$. We denote the convex subset of continuous positive definite functions of norm one by $\mc{P}_1(G)$.

The adjoint of the multiplication $m:A(G)\hten A(G)\rightarrow A(G)$ defines a co-associative co-multiplication $\Gam:VN(G)\rightarrow VN(G\times G)$, where we have used the fact that $VN(G\times G)=VN(G)\oten VN(G)=(A(G)\hten A(G))^*$ \cite[Theorem 7.2.4]{ER}, and $\oten$ denotes the von Neumann algebra tensor product. This co-multiplication is symmetric in the sense that $\Gam=\Sigma\circ\Gam$, where $\Sigma:VN(G\times G)\rightarrow VN(G\times G)$ is the flip map; it satisfies $\Gam(\lm(s))=\lm(s)\ten\lm(s)$, $s\in G$, and can be written as
$$\Gam(x)=V(x\ten 1)V^*, \ \ \ x\in VN(G),$$
where $V$ is the unitary in $\LI\oten VN(G)$ given by
$$V\xi(s,t)=\xi(s,s^{-1}t), \ \ \ s,t\in G, \ \xi\in L^2(G\times G).$$
The co-associativity of $\Gam$ translates into the following \e{pentagonal relation} for $V$:
\begin{equation}\label{e:pentagonal} V_{12}V_{13}V_{23}=V_{23}V_{12},\end{equation}
where $V_{12}=V\ten 1$, $V_{23}=1\ten V$, $V_{13}=(\sigma\ten 1)V_{23}(\sigma\ten 1)$, and $\sigma$ is the flip map on $L^2(G\times G)$.

The group von Neumann algebra $VN(G)$ becomes an operator $A(G)$-bimodule in the canonical fashion, and the bimodule actions can be written in terms of the co-multiplication:
$$u\cdot x = x\cdot u = (\id\ten u)\Gam(x) = (u\ten\id)\Gam(x), \ \ \ u\in A(G), \ x\in VN(G).$$
It follows that $VN(G)$ is faithful as a left/right operator $A(G)$-module (respectively, $A(G)$-bimodule), and that under the isomorphism $\mc{CB}(A(G),VN(G))\cong VN(G\times G)$, the canonical morphism $\Delta_{VN(G)}=\Gam$.

Given a closed subgroup $H\leq G$, we let $I(H)=\{u\in A(G)\mid u|_{H}\equiv 0\}$ denote the closed ideal of functions in $A(G)$ which vanish on $H$. By the proof of \cite[Proposition 1.7]{ARS} $I(H)$ is an essential ideal. It follows from \cite{Herz} that the restriction $r:A(G)\twoheadrightarrow A(H)$ is a complete quotient map with kernel $I(H)$, therefore $A(H)\cong A(G)/I(H)$.

\section{Relative flatness and inner amenability}

If $G$ is a locally compact group and $p\in[1,\infty]$, then $G$ acts by conjugation on $L^p(G)$ via
$$\beta_p(s)f(t)=f(s^{-1}ts)\Delta(s)^{1/p}, \ \ \ s,t\in G, \ f\in L^p(G).$$
When $p=2$, we obtain a strongly continuous unitary representation $\beta_2:G\rightarrow\BLT$ satisfying $\beta_2(s)=\lm(s)\rho(s)$ for $s\in G$, and when $p=\infty$, the conjugation action becomes
$$\beta_\infty(s)f(t)=f(s^{-1}ts), \ \ \ s,t\in G, \ f\in\LI.$$
Following Paterson \cite[2.35.H]{Pat2}, we say that $G$ is \e{inner amenable} if there exists a state $m\in\LI^*$ satisfying
\begin{equation}\label{e:PatIA}\la m,\beta_{\infty}(s)f\ra=\la m,f\ra \ \ \ s\in G, \ f\in\LI.\end{equation}

\begin{remark} In \cite{Effros}, Effros defined a discrete group $G$ to be ``inner amenable'' if there exists a conjugation invariant mean $m\in\ell^{\infty}(G)^*$ such that $m\neq\delta_e$. In what follows, inner amenability will always refer to the definition (\ref{e:PatIA}) given above.\end{remark}

The class of inner amenable locally compact groups forms a large, interesting class of groups containing all amenable groups and IN groups, where a locally compact group $G$ is IN if there exists a compact neighbourhood of the identity which is invariant under conjugation. For example, compact, abelian and discrete groups are IN, and therefore inner amenable.

A strongly continuous unitary representation $\pi:G\rightarrow\mc{B}(H_\pi)$ of a locally compact group $G$ is said to be \e{amenable} if there exists a state $m_\pi\in\mc{B}(H_\pi)^*$ such that
$$\la m_\pi,\pi(s)^*T\pi(s)\ra=\la m_\pi,T\ra, \ \ \ \forall \ s\in G, \ T\in\mc{B}(H_\pi).$$
This concept was introduced by Bekka \cite{Bekka}, who showed, among other things, that $G$ is inner amenable precisely when $\beta_2$ is an amenable unitary representation \cite[Theorem 2.4]{Bekka}. By \cite[Proposition 3.1]{St1}, inner amenability is equivalent to the existence of a $\beta_2$-invariant state on $\beta_2(G)''\subseteq\BLT$, the von Neumann subalgebra generated by the conjugate representation. We now show that inner amenability is equivalent to the existence of a $\beta_2$-invariant state on $\LG$, i.e., a $G$-invariant state under the canonical $G$-action:
$$x\lhd s=\lm(s)^*x\lm(s), \ \ \ x\in\LG, \ s\in G.$$
In turn, we answer a question raised by Lau and Paterson in \cite[Example 5]{LP2}.

\begin{prop}\label{p:IA} A locally compact group $G$ is inner amenable if and only if there exists a $G$-invariant state on $\LG$.\end{prop}

\begin{pf} If $G$ is inner amenable, then by \cite[Theorem 2.4]{Bekka} there exists a $\beta_2$-invariant state $m\in\BLT^*$, whose restriction to $\LG$ is necessarily $G$-invariant, as
$$\la m,\lm(s)^*x\lm(s)\ra=\la m,\lm(s)^*\rho(s)^*x\rho(s)\lm(s)\ra=\la m,\beta_2(s)^*x\beta_2(s)\ra=\la m,x\ra$$
for all $x\in VN(G)$, $s\in G.$

Conversely, suppose $m\in \LG^*$ is a $G$-invariant state. Since $\LG$ is standardly represented on $\LT$, there exists a net of unit vectors $(\xi_\alpha)$ in $\mc{P}$ such that $(\om_{\xi_\alpha})$ converges to $m$ in the weak* topology of $\LG^*$. By $G$-invariance, it follows that
\begin{equation*}\beta_2(s)\cdot\om_{\xi_\alpha}\cdot\beta_2(s)^*-\om_{\xi_\alpha}=\om_{\beta_2(s)\xi_\alpha}-\om_{\xi_\alpha}\rightarrow 0\end{equation*}
weakly in $A(G)=\LG_*$ for all $s\in G$. By the standard convexity argument, we obtain a net of unit vectors $(\eta_\gamma)$ in $\mc{P}$ satisfying
\begin{equation*}\norm{\beta_2(s)\cdot\om_{\eta_\gamma}\cdot\beta_2(s)^*-\om_{\eta_\gamma}}_{A(G)}=\norm{\om_{\beta_2(s)\eta_\gamma}-\om_{\eta_\gamma}}_{A(G)}\rightarrow 0, \ \ \ s\in G.\end{equation*}
However, since $\beta_2(s)=\lm(s)\rho(s)=\lm(s)J\lm(s)J$ we have $\beta_2(s)\mc{P}\subseteq\mc{P}$ for any $s\in G$ by \cite[Theorem 1.1]{H}. Then \cite[Lemma 2.10]{H} entails
\begin{equation*}\norm{\beta_2(s)\eta_\gamma-\eta_\gamma}_{\LT}^2\leq\norm{\om_{\beta_2(s)\eta_\gamma}-\om_{\eta_\gamma}}_{A(G)}\rightarrow0, \ \ \ s\in G.\end{equation*}
Letting $f_\gamma:=|\eta_\gamma|^2$, we obtain a net of states in $\LO$ satisfying
$$\norm{\beta_1(s)f_\gamma - f_\gamma}_{\LO}=\norm{\om_{\beta_2(s)\eta_\gamma}-\om_{\eta_\gamma}}_{\LO}\leq 2\norm{\beta_2(s)\eta_\gamma - \eta_\gamma}_{\LT}\rightarrow0, \ \ \ s\in G.$$
Any weak* cluster point $M\in\LI^*$ of $(f_\gamma)$ will therefore be conjugate invariant, and $G$ is inner amenable.
\qed\end{pf}

As an immediate corollary, we obtain the following hereditary property of inner amenability, which appears to be new.

\begin{cor}\label{c:subgroup} Let $G$ be a locally compact group and let $H$ be a closed subgroup of $G$. If $G$ is inner amenable, then $H$ is inner amenable.
\end{cor}

\begin{pf} Let $VN_H(G):=\{\lm_G(s)\mid s\in H\}''\subseteq\LG$. Then the map $i_H:VN(H)\rightarrow VN_H(G)$ given by
$$i_H(\lm_H(s))=\lm_G(s), \ \ \ s\in H,$$
is a *-isomorphism of von Neumann algebras. Thus, if $m\in\LG^*$ is a $G$-invariant state then $m_H:=m|_{VN_H(G)}\circ i_H\in VN(H)^*$ is an $H$-invariant state on $VN(H)$, so $H$ is inner amenable by Proposition \ref{p:IA}.
\qed\end{pf}


In \cite[Corollary 3.2]{LP1}, Lau and Paterson proved the following equivalence for a locally compact group $G$:
\begin{enumerate}
\item $G$ is amenable;
\item $G$ is inner amenable and $\LG$ is $1$-injective in $\C-\mathbf{mod}$.
\end{enumerate}
The following theorem will allow us to describe the above equivalence from a purely homological perspective, elucidating the relationship between amenability and inner amenability.

\begin{thm}\label{t:IA=relinj}   A locally compact group $G$ is inner amenable if and only if $A(G)$ is relatively $1$-flat in $\mathbf{mod}-A(G)$.
\end{thm}

\begin{pf} If $G$ is inner amenable, then by \cite[Proposition 1.13]{St2} there exists a net of states $(f_\alpha)$ in $\LO$ satisfying
$$\norm{\beta_1(s)f_\alpha-f_\alpha}_{\LO}\rightarrow0, \ \ \ s\in G,$$
uniformly on compact sets. The square roots $\xi_\alpha:=f_\alpha^{1/2}\in\LT$ then satisfy
$$\norm{\beta_2(s)\xi_\alpha - \xi_\alpha}^2_{\LT}\leq\norm{\beta_1(s)f_\alpha-f_\alpha}_{\LO}\rightarrow0, \ \ \ s\in G,$$
uniformly on compact sets. Thus, combining \cite[Lemma 3.1, Lemma 4.1]{RX}, it follows that $\Gam:\LG\rightarrow VN(G\times G)$ has a completely contractive left inverse $\Phi$ which is a left $A(G)$-module map. Since $\LG$ is faithful in $A(G)-\mathbf{mod}$, this entails the relative 1-injectivity of $\LG$ in $A(G)-\mathbf{mod}$, and hence, the relative 1-flatness of $A(G)$ in $\mathbf{mod}-A(G)$.

Conversely, relative $1$-flatness of $A(G)$ in $\mathbf{mod}-A(G)$ implies the existence of a completely contractive morphism $\Phi: VN(G\times G)\rightarrow \LG$ satisfying $\Phi\circ\Gam=\id_{\LG}$.
It follows that $\Gam\circ\Phi: VN(G\times G)\rightarrow  VN(G\times G)$ is a projection of norm one onto the image of $\Gam$. Thus, by \cite{To}, $\Gam\circ\Phi$ is a $\Gam(\LG)$-bimodule map, which by injectivity of $\Gam$ yields the identity
\begin{equation}\label{2} x\Phi(T)y=\Phi(\Gam(x)T\Gam(y))\end{equation}
for all $x,y\in\LG$ and $T\in VN(G\times G)$.

For $x\in VN(G)$, the module property of $\Phi$ implies $u\cdot\Phi(x\ten 1)=\Phi(x\ten u\cdot1)=u(e)\Phi(x\ten 1)$ for all $u\in A(G)$. The standard argument then shows $\Phi(x\ten1)\in\C1$, so that $m:\LG\rightarrow\C$ defined by $\la m,x\ra=\Phi(x\ten 1)$, $x\in \LG$, yields a state on $\LG$. Moreover, by equation (\ref{2}) we obtain
\begin{align*}\la m,\lm(s)x\lm(s)^*\ra&=\Phi(\lm(s)x\lm(s)^*\ten 1)=\Phi((\lm(s)\ten\lm(s))(x\ten 1)(\lm(s)^*\ten\lm(s)^*))\\
&=\Phi(\Gam(\lm(s))(x\ten 1)\Gam(\lm(s)^*))=\lm(s)\Phi(x\ten 1)\lm(s)^*=\Phi(x\ten 1)\\
&=\la m,x\ra\end{align*}
for any $x\in \LG$ and $s\in G$. Thus, $m$ is a $G$-invariant state on $\LG$, which by Proposition \ref{p:IA} implies that $G$ is inner amenable.
\qed\end{pf}

Combining Theorem \ref{t:IA=relinj} with \cite[Corollary 5.3]{C}, we can now recast the Lau--Paterson equivalence in purely homological terms:
\begin{enumerate}
\item $\LG$ is $1$-injective in $A(G)-\mathbf{mod}$;
\item $\LG$ is relatively $1$-injective in $A(G)-\mathbf{mod}$ and $1$-injective in $\C-\mathbf{mod}$.
\end{enumerate}

Let $G$ be a locally compact group. A function $u\in B(G\times G)$ is said to be \textit{properly supported}, if for every compact subset $K\subseteq G$, the sets $\mathrm{supp}(u)\cap G\times K$ and $\mathrm{supp}(u)\cap K\times G$ are compact \cite[Definition 4.2]{Claire}. The group $G$ is said to have Property $(W)$ if for every compact set $K\subseteq G$ and every $\varepsilon>0$, there exists a properly supported bounded positive definitive function $u\in B(G\times G)$ such that $|u(s,s)-1|<\varepsilon$ for all $s\in K$ \cite[Definition 4.3]{Claire}. This notion was introduced to study the relationship between amenable actions of locally compact groups and exactness of reduced group $C^*$-algebras. It was shown that every inner amenable group has Property $(W)$ \cite[Proposition 4.6]{Claire}, but it was left open whether they are equivalent \cite[Problem 9.1]{Claire}. We now show that this is indeed the case.

\begin{thm} A locally compact group $G$ is inner amenable if and only if it has Property $(W)$.\end{thm}

\begin{pf} Suppose $G$ has Property $(W)$, witnessed by a net $(u_\alpha)$ of properly supported positive definite functions in $B(G\times G)$ satisfying
$$|u_\alpha(s,s)-1|\rightarrow 0$$
for $s\in G$, uniformly on compact sets. Without loss of generality we may assume that $u_\alpha(e,e)=1$ for all $\alpha$. By Nielson's lemma \cite[Lemma 10.3]{N} (see also \cite[Proposition 5.1]{dH}) it follows that
$$(u_\alpha)|_{G_\Delta}\cdot v\rightarrow v, \ \ \ v\in A(G).$$
Moreover, since $u_\alpha$ is properly supported, for any $v\in A(G)$ with compact support, the function $u_\alpha\cdot(1\ten v)\in B(G\times G)$ is compactly supported, and hence lies in $A(G\times G)$. Thus, $u_\alpha\cdot(1\ten v)\in A(G\times G)$ for all $v\in A(G)$, and
$$\norm{[u_\alpha\cdot(1\ten v_{ij})]}_{M_n(A(G\times G))}=\norm{[u_\alpha\cdot(1\ten v_{ij})]}_{M_n(B(G\times G))}\leq\norm{u_{\alpha}}_{B(G\times G)}\norm{[v_{ij}]}_{M_n(A(G))},$$
so that $\norm{u_\alpha}_{\mc{CB}(A(G),A(G\times G))}\leq\norm{u_\alpha}_{B(G\times G)}=1$. Define maps $\Phi_\alpha:VN(G\times G)\rightarrow VN(G)$ by
$$\la\Phi_\alpha(X),v\ra=\la u_\alpha\cdot(1\ten v),X\ra, \ \ \ X\in VN(G\times G), \ v\in A(G).$$
Then $\norm{\Phi_\alpha}_{cb}\leq\norm{u_\alpha}_{\mc{CB}(A(G),A(G\times G))}\leq 1$, and
$$\la\Phi_\alpha(u\cdot X),v\ra=\la u_\alpha\cdot(1\ten v),u\cdot X\ra=\la u_\alpha\cdot(1\ten vu), X\ra=\la\Phi_\alpha(X),vu\ra=\la u\cdot\Phi_\alpha(X),v\ra$$
for all $X\in VN(G\times G)$ and $u,v\in A(G)$. Passing to a subnet if necessary, we may assume that $(\Phi_\alpha)$ converges weak* to $\Phi\in\mc{CB}(VN(G\times G),VN(G))=(VN(G\times G)\hten A(G))^*$. Then 
$$\la \Phi(\Gamma(x)),v\ra=\lim_\alpha\la u_\alpha\cdot(1\ten v),\Gamma(x)\ra=\lim_\alpha\la (u_\alpha)|_{G_\Delta}\cdot v,x\ra=\lim_\alpha\la v,x\ra=\la x,v\ra$$
for all $x\in VN(G)$ and $v\in A(G)$. Hence, $\Phi:VN(G\times G)\rightarrow VN(G)$ is a completely contractive left $A(G)$-module inverse to $\Gam$, entailing the relative 1-flatness of $A(G)$ in $\mathbf{mod}-A(G)$, and therefore the inner amenability of $G$ by Theorem \ref{t:IA=relinj}   .
\qed\end{pf}

At present, we believe but have been unable to show that inner amenability of $G$ is equivalent to relative $C$-flatness of $A(G)$ in $\mathbf{mod}-A(G)$ for $C>1$. We can, however, provide a number of examples which support the conjecture based on the following proposition.

\begin{prop}\label{p:C-inj} Let $G$ be a locally compact group and let $H$ be a closed subgroup. If $\LG$ is $C$-injective in $A(G)-\mathbf{mod}$ for $C\geq 1$, then $VN(H)$ is $C$-injective in $A(H)-\mathbf{mod}$.
\end{prop}

\begin{pf} Let $r:A(G)\twoheadrightarrow A(H)$ be the complete quotient map given by restriction. Then $\mc{B}(L^2(H))$ becomes a left $A(G)$-module via
$$u\cdot T=(\id\ten r(u))\Gam^r(T), \ \ \ u\in A(G), \ T\in\mc{B}(L^2(H)),$$
where $\Gam^r:\mc{B}(L^2(H))\rightarrow\mc{B}(L^2(H))\oten VN(H)$ is the canonical lifting of the co-multiplication on $VN(H)$, given by
$$\Gam^r(T)=V(T\ten 1)V^*, \ \ \ T\in\mc{B}(L^2(H)).$$
Clearly, $VN(H)$ is a closed $A(G)$-submodule of $\mc{B}(L^2(H))$. Hence, the inclusion $VN(H)\hookrightarrow\LG$ extends to a morphism $E:\mc{B}(L^2(H))\rightarrow\LG$ with $\norm{E}_{cb}\leq C$. We show that $E(\mc{B}(L^2(H)))=VN(H)$. To this end, fix $T\in \mc{B}(L^2(H))$. Then for $u\in A(G)$ and $v\in I(H)$, we have
$$\la E(T),u\cdot v\ra=\la v\cdot E(T),u\ra=\la E(v\cdot T),u\ra=0$$
as $r(v)=0$. Since $I(H)$ is essential it follows that $E(T)\in I(H)^{\perp}=VN(H)$. Thus, $E:\mc{B}(L^2(H))\rightarrow VN(H)$ is a completely bounded $A(H)$-module projection with $\norm{E}_{cb}\leq C$. Since $VN(H)$ has an $A(H)$-invariant state $m\in VN(H)^*$ satisfying
$$\la m,u\cdot x\ra=u(e)\la m,x\ra, \ \ \ u\in A(H), \ x\in VN(H),$$
it follows that $VN(H)$ is an amenable quantum group, and the proof of \cite[Theorem 5.5]{CN} implies that $\mc{B}(L^2(H))$ is $1$-injective in $A(H)-\mathbf{mod}$. Thus, $VN(H)$ is $C$-injective in $A(H)-\mathbf{mod}$.
\qed\end{pf}

\begin{cor}\label{c:C-inj} Let $G$ be a locally compact group such that $\LG$ is $C$-injective in $A(G)-\mathbf{mod}$ for some $C\geq 1$. Then every closed inner amenable subgroup of $G$ is amenable.
\end{cor}

\begin{pf} By Proposition \ref{p:C-inj} we know that $VN(H)$ is $C$-injective in $A(H)-\mathbf{mod}$ for any closed subgroup $H$. Hence, there exists a completely bounded projection $E:\mc{B}(L^2(H))\rightarrow VN(H)$, which, by \cite[Theoerem 3.1]{CS} (see also \cite{Pi}) implies that $VN(H)$ is an injective von Neumann algebra. If $H$ is inner amenable, then by \cite[Corollary 3.2]{LP1} it is necessarily amenable.
\qed\end{pf}

\begin{cor}\label{c:F_2} Let $G$ be a locally compact group containing $\F_2$ as a closed subgroup and for which $\LG$ is $1$-injective in $\C-\mathbf{mod}$. Then $\LG$ is not relatively $C$-injective in $A(G)-\mathbf{mod}$ for any $C\geq1$.
\end{cor}

\begin{pf} If $\LG$ were relatively $C$-injective in $A(G)-\mathbf{mod}$, then it would be $C$-injective in $A(G)-\mathbf{mod}$ by \cite[Proposition 2.3]{C}. Since $\F_2$ is inner amenable, Corollary \ref{c:C-inj} would imply that it is amenable, which is absurd.
\qed\end{pf}

Since almost connected groups have injective von Neumann algebras (see \cite{Pat3} and the references therein), and are non-amenable precisely when they contain $\F_2$ has a closed subgroup \cite[Theorem 5.5]{Rick}, Corollary \ref{c:F_2} implies that any non-amenable almost connected group $G$ cannot have a relatively $C$-flat (and hence $C$-biflat) Fourier algebra for any $C\geq 1$. In particular, $A(SL(n,\R))$, $A(SL(n,\C))$ and $A(SO(1,n))$ are not relatively flat (or biflat) for $n\geq 2$. This result builds on the analysis of \cite[\S4]{ARS}, where it was suspected that $A(SL(3,\C))$ would fail to be relatively biflat.

Regarding the relative projectivity of $A(G)$, we now establish the converse to \cite[Lemma 3.2]{RX}, providing a partial solution to the open question of relative projectivity of $A(G)$ \cite[\S4]{FLS}.

\begin{prop} Let $G$ be a locally compact group. Then $A(G)$ is relatively $1$-projective in $\mathbf{mod}-A(G)$ if and only if $G$ is an IN group.
\end{prop}

\begin{pf} Assuming relative $1$-projectivity of $A(G)$ in $\mathbf{mod}-A(G)$, there exists a normal completely contractive left $A(G)$-module map $\Phi:VN(G\times G)\rightarrow\LG$ such that $\Phi\circ\Gam=\id_{\LG}$. By the proof of Theorem \ref{t:IA=relinj}    we obtain a normal $G$-invariant state on $\LG$, which, by \cite[Proposition 4.2]{Tay} implies that $G$ is IN. The converse follows from \cite[Lemma 3.2]{RX}.
\qed\end{pf}

\section{Relative biflatness of $A(G)$}\label{s:relbiflat}

Given a locally compact group $G$ and a closed subgroup $H$, a bounded net $(\vphi_\alpha)$ in $B(G)$ is called an \e{approximate indicator} for $H$ \cite[Definition 2.1]{ARS} if
\begin{enumerate}
\item $\lim_{\alpha} (\vphi_{\alpha}|_{H})\cdot u=u$ for all $u\in A(H)$;
\item $\lim_{\alpha} \vphi_{\alpha}\cdot v=0$ for all $v\in I(H)$.
\end{enumerate}
If $\norm{\vphi_\alpha}_{B(G)}\leq 1$ for all $\alpha$ we say that $(\vphi_\alpha)$ is a \e{contractive approximate indicator} for $H$.

In \cite[Proposition 2.3]{ARS} it was shown that $A(G)$ is relatively $C$-biflat if the diagonal subgroup $G_\Delta\leq G\times G$ has an approximate indicator $(\vphi_{\alpha})$ with $\norm{\vphi_{\alpha}}_{B(G)}\leq C$. We now establish the converse when $C=1$, which is one of the main results of the paper.

\begin{thm}\label{t:biflat} Let $G$ be a locally compact group. Then $A(G)$ is relatively 1-biflat if and only if $G_\Delta$ has a contractive approximate indicator.
\end{thm}

\begin{pf} We need only establish necessity. Consider the right $\LO$-action on $\LG$ given by
$$x\lhd f=\int_G\lm(s)^*x\lm(s)f(s)ds, \ \ \ x\in\LG, \ f\in\LO.$$
For $f\in\LO$, we let $\h{\Theta}(f):\LG\rightarrow\LG$ and $\h{\theta}_f:VN(G\times G)\rightarrow VN(G\times G)$ be the normal completely bounded maps given respectively by $\h{\Theta}(f)(x)=x\lhd f$, $x\in\LG$, and
$$\h{\theta}_f(X)=\int_G (\lm(s)^*\ten\lm(s)^*)X(\lm(s)\ten\lm(s))f(s)ds, \ \ \ X\in VN(G\times G).$$
Relative 1-biflatness of $A(G)$ implies the existence of a completely contractive $A(G)$-bimodule left inverse $\Phi:VN(G\times G)\rightarrow\LG$ to $\Gam$. It follows as in Theorem \ref{t:IA=relinj}    that $\Gam\circ\Phi$ is a $\Gam(\LG)$-bimodule map. By Wittstock's bimodule extension theorem \cite{Witt2}, this map extends to an $\Gam(\LG)$-bimodule map $\Psi:\mc{B}(L^2(G\times G))\rightarrow\mc{B}(L^2(G\times G))$. Moreover, \cite[Lemma 2.3]{MNW} allows us to approximate $\Psi$ in the point weak* topology by a net $(\Psi_\alpha)$ of normal completely bounded $\Gam(\LG)$-bimodule maps. Thus, for any $X\in VN(G\times G)$, we have
\begin{align*}\Gam\circ\Phi(\h{\theta}_f(X))&=\Psi(\h{\theta}_f(X))=\Psi\bigg(\int_G (\lm(s)^*\ten\lm(s)^*)X(\lm(s)\ten\lm(s))f(s)ds\bigg)\\
&=\lim_\alpha\Psi_\alpha\bigg(\int_G (\lm(s)^*\ten\lm(s)^*)X(\lm(s)\ten\lm(s))f(s)ds\bigg)\\
&=\lim_\alpha\bigg(\int_G \Psi_\alpha((\lm(s)^*\ten\lm(s)^*)X(\lm(s)\ten\lm(s)))f(s)ds\bigg)\\
&=\lim_\alpha\bigg(\int_G \Psi_\alpha(\Gam(\lm(s)^*)X\Gam(\lm(s)))f(s)ds\bigg)\\
&=\lim_\alpha\bigg(\int_G \Gam(\lm(s)^*)\Psi_\alpha(X)\Gam(\lm(s))f(s)ds\bigg)\\
&=\lim_\alpha\h{\theta}_f(\Psi_\alpha(X))=\h{\theta}_f(\Psi(X))=\h{\theta}_f(\Gam\circ\Phi(X)),
\end{align*}
where we used normality of $\Psi_\alpha$ and $\h{\theta}_f$ in the fourth and eighth equality, respectively. By definition of $\h{\theta}_f$, we have $\h{\theta}_f\circ\Gam=\Gam\circ\h{\Theta}(f)$, so the above calculation entails $\Gam\circ\Phi\circ\h{\theta}_f=\Gam\circ\h{\Theta}(f)\circ\Phi$, which, by injectivity of $\Gam$, implies $\Phi\circ\h{\theta}_f=\h{\Theta}(f)\circ\Phi$.

As in the proof of Theorem \ref{t:IA=relinj}   , the restriction $\Phi|_{\LG\ten 1}$ defines a state $m\in\LG^*$. The bimodule property of $\Phi$ ensures that $m$ is invariant for the $A(G)$-action on $\LG$, that is,
$$\la m,u\cdot x\ra=u(e)\la m,x\ra, \ \ \ x\in\LG, \ u\in A(G).$$
Moreover, for $f\in\LO$ and $x\in\LG$ we have
\begin{align*}\la m,x\lhd f\ra&=\Phi\bigg(\int_G (\lm(s)^*\ten\lm(s)^*)(x\ten1)(\lm(s)\ten\lm(s))f(s)ds\bigg)\\
&=\Phi(\h{\theta}_f(x\ten 1))\\
&=\h{\Theta}(f)(\Phi(x\ten 1))\\
&=\la f,1\ra\la m,x\ra.
\end{align*}
Approximating $m\in VN(G)^*$ in the weak* topology by a net of states $(u_\beta)$ in $A(G)$, it follows that
$$u_\beta\cdot v-v(e)u_\beta\rightarrow 0 \ \ \textnormal{and} \ \ f\lhd u_\beta-\la f,1\ra u_\beta\rightarrow 0$$
weakly in $A(G)$ for all $v\in A(G)$ and $f\in\LO$, where $f\lhd u_\beta=(\h{\Theta}(f))_*(u_\beta)$. By the standard convexity argument, we obtain a net of states $(u_\gamma)$ in $A(G)$ satisfying
\begin{equation}\label{e:L1inv}\norm{u_\gamma\cdot v-v(e)u_\gamma}_{A(G)}, \ \norm{f\lhd u_\gamma-\la f,1\ra u_\gamma}_{A(G)}\rightarrow0, \ \ \ v\in A(G), \ f\in\LO.\end{equation}

For $s\in G$ and $v\in A(G)$ we define $s\lhd v\in A(G)$ by $s\lhd v(t)=v(s^{-1}ts)$, $t\in G$. Then by left invariance of the Haar measure it follows that
\begin{equation}\label{e:inv}s\lhd(f\lhd v)=(l_sf)\lhd v, \ \ \ s\in G, \ f\in\LO, \ v\in A(G),\end{equation}
where $l_sf(t)=f(st)$, $s,t\in G$. Fix a state $f_0\in\LO$, and consider the net $(f_0\lhd u_\gamma)$. For $\varepsilon>0$, take a neighbourhood $U$ of the identity $e\in G$ such that
$$\norm{l_sf_0-f_0}_{\LO}<\frac{\varepsilon}{2},  \ \ \ s\in U.$$
Then for any compact set $K\subseteq G$, there exist $s_1,...,s_n\in K$ such that $K\subseteq\cup_{i=1}^n Us_i$. Take $\gamma_\varepsilon$ such that for $\gamma\geq\gamma_\varepsilon$
\begin{equation*}\norm{(l_{s_i}f_0)\lhd u_\gamma-u_\gamma}_{A(G)}<\frac{\varepsilon}{4}, \ \ \ 1\leq i\leq n.\end{equation*}
Applying (\ref{e:inv}) together with the $\LO$-invariance in (\ref{e:L1inv}), it follows by the standard argument (see \cite[Lemma 7.1.1]{R1}) that
$$\norm{k\lhd(f_0\lhd u_\gamma)-f_0\lhd u_\gamma}_{A(G)}<\varepsilon, \ \ \ k\in K.$$
Hence, the net $(f_0\lhd\psi_\gamma)$ satisfies
$$\norm{s\lhd(f_0\lhd u_\gamma)-f_0\lhd u_\gamma}_{A(G)}\rightarrow 0, \ \ \ s\in G,$$
uniformly on compact sets. Using both the $A(G)$ and $\LO$-invariance from equation (\ref{e:L1inv}), a $3\varepsilon$-argument also shows that
$$\norm{(f_0\lhd u_\gamma)\cdot v-v(e)f_0\lhd u_\gamma}_{A(G)}\rightarrow 0, \ \ \ v\in A(G).$$
Forming $|f_0\lhd u_\gamma|^2$, we may further assume $f_0\lhd u_\gamma(s)\geq0$ for all $s\in G$, as one may easily verify using boundedness and multiplicativity of the $G$-action that
$$\norm{u\cdot|f_0\lhd u_\gamma|^2-u(e)|f_0\lhd u_\gamma|^2}_{A(G)}, \ \norm{s\lhd|f_0\lhd u_\gamma|^2-|f_0\lhd u_\gamma|^2}_{A(G)}\rightarrow0$$
for all $u\in A(G)$ and for all $s\in G$, uniformly on compact sets.

Now, since $\LG$ is standardly represented on $\LT$, there exist unit vectors $\xi_\gamma\in\mc{P}$ satisfying
$$\om_{\xi_\gamma}|_{\LG}=f_0\lhd u_\gamma.$$
Note that $J\xi_\gamma=\xi_\gamma$ and that $\xi_\gamma$ is necessarily real-valued by uniqueness. For any $s\in G$ we have $s\lhd\om_{\xi_\gamma}=\om_{\beta_2(s)\xi_\gamma}$ and $\beta_2(s)\mc{P}\subseteq\mc{P}$. Thus \cite[Lemma 2.10]{H} implies
\begin{equation}\label{e:uni}\norm{\beta_2(s)\xi_\gamma-\xi_\gamma}_{\LT}^2\leq\norm{\om_{\beta_2(s)\xi_\gamma}-\om_{\xi_\gamma}}_{A(G)}= \norm{s\lhd u_\gamma- u_\gamma}_{A(G)}\rightarrow0\end{equation}
for all $s\in G$, uniformly on compact sets.

Define the function $\vphi_\gamma\in\mc{P}_1(G\times G)\subseteq B(G\times G)$ by
$$\vphi_\gamma(s,t)=\la\lm(s)\rho(t)\xi_\gamma,\xi_\gamma\ra, \ \ \ s,t\in G,$$
and consider the associated normal completely positive map $\Theta(\vphi_\gamma)\in\mc{CB}_{A(G\times G)}(VN(G\times G))$ given by
$$\Theta(\vphi_\gamma)(\lm(s)\ten\lm(t))=\vphi_\gamma(s,t)\lm(s)\ten\lm(t), \ \ \ s,t\in G.$$
We claim that the bounded net $(\Theta(\vphi_\gamma))$ clusters to a completely positive $A(G\times G)$-module projection $VN(G\times G)\rightarrow VN(G_\Delta)$.

To verify the claim, first consider the net $(\om_{\xi_\gamma})$ in $\TC=\BLT_*$. By passing to a subnet we may assume that $(\om_{\xi_\gamma})$ converges weak* to a state $M\in\BLT^*$. For each $\gamma$ define the unital completely positive map $\Phi_\gamma:VN(G\times G)\rightarrow VN(G)$ by
$$\Phi_\gamma(X)=(\id\ten\om_{\xi_\gamma})V(1\ten U)X(1\ten U)V^*, \ \ \ X\in VN(G\times G),$$
where $U$ is the self-adjoint unitary given by $U=\h{J}J$, and $\h{J}$ is complex conjugation on $\LT$. Since $\Gam(x)=V(x\ten 1)V^*$, $x\in VN(G)$, and $UVN(G)U=VN(G)'$, one easily sees that the range of $\Phi_\gamma$ is indeed contained in $VN(G)$.

For every $\gamma$ and $s,t\in G$, we have
\begin{align*}\Theta(\vphi_\gamma)(\lm(s)\ten\lm(t))&=\la\lm(s)\rho(t)\xi_\gamma,\xi_\gamma\ra\lm(s)\ten\lm(t)\\
&=(\id\ten\id\ten\om_{\xi_\gamma})(\lm(s)\ten\lm(t)\ten\lm(s)\rho(t))\\
&=(\id\ten\id\ten\om_{\xi_\gamma})(\lm(s)\ten1\ten\lm(s))(1\ten\lm(t)\ten\rho(t))\\
&=(\id\ten\id\ten\om_{\xi_\gamma})(\lm(s)\ten1\ten\lm(s))(1\ten(1\ten U)V(\lm(t)\ten 1)V^*(1\ten U))\\
&=(\id\ten\id\ten\om_{\xi_\gamma})(\lm(s)\ten1\ten1)(1\ten(1\ten U)V(\lm(t)\ten\rho(s))V^*(1\ten U))\\
&=(\id\ten\id\ten\om_{\xi_\gamma})(\lm(s)\ten1\ten1)(1\ten V(\lm(t)\ten\rho(s))V^*) \ \ \ \ (\textnormal{as $U\xi_\gamma=\xi_\gamma$})\\
&=(\id\ten\id\ten\om_{\xi_\gamma})(\lm(s)\ten1\ten1)(1\ten V((1\ten U)(\lm(t)\ten\lm(s))(1\ten U))V^*)\\
&=\lm(s)\ten\Phi_\gamma(\lm(t)\ten\lm(s))\\
&=\lm(s)\ten\Phi_\gamma(\Sigma(\lm(s)\ten\lm(t)))\\
&=(\id\ten\Phi_\gamma\circ\Sigma)(\lm(s)\ten\lm(s)\ten\lm(t)))\\
&=(\id\ten\Phi_\gamma\circ\Sigma)(\Gam\ten\id)(\lm(s)\ten\lm(t)).\end{align*}
By normality we see that $\Theta(\vphi_\gamma)=(\id\ten\Phi_\gamma\circ\Sigma)(\Gam\ten\id)$. Since $(\Phi_\gamma)$ is bounded, it follows that $(\Phi_\gamma)$ converges in the stable point weak* topology to the map $\Phi_M\in\mc{CB}(VN(G\times G),VN(G))$ given by
$$\Phi_M(X)=(\id\ten M)V(1\ten U)X(1\ten U)V^*, \ \ \ X\in VN(G\times G).$$
Hence, the net $(\Theta(\vphi_\gamma))$ converges weak* to a map $\Theta\in\mc{CB}(VN(G\times G))$ satisfying $$\Theta=(\id\ten\Phi_M\circ\Sigma)(\Gam\ten\id).$$
If $\Phi_M$ were a left $A(G)$-module left inverse to $\Gam$, it would follow that $\Theta=\Gam\circ\Phi_M\circ\Sigma$, hence the claim. We therefore turn to the required properties of $\Phi_M$.

First, let $\h{V}$ be the unitary in $VN(G)'\oten\LI$ given by
$$\h{V}\zeta(s,t)=\zeta(st,t)\Delta(t)^{1/2}, \ \ \ \zeta\in L^2(G\times G), \ s,t\in G.$$
Then, for $\eta\in\LT$, the compact convergence (\ref{e:uni}) entails
$$\norm{V\sigma \h{V}\sigma\eta\ten\xi_\gamma - \eta\ten\xi_\gamma}^2_{L^2(G\times G)}=\int_G\int_G |\eta(s)|^2|\beta_2(s)\xi_\gamma(t)-\xi_\gamma(t)|^2dsdt\rightarrow 0.$$
Noting that $\h{V}=\sigma(1\ten U)V(1\ten U)\sigma$, for $X\in VN(G\times G)$ we therefore have
\begin{align*}\la\Phi_M(X),\om_{\eta}\ra&=\lim_\gamma\la V(1\ten U)X(1\ten U)V^*\eta\ten\xi_\gamma,\eta\ten\xi_\gamma\ra\\
&=\lim_\gamma\la(1\ten U)V(1\ten U)X(1\ten U)V^*(1\ten U)\eta\ten\xi_\gamma,\eta\ten\xi_\gamma\ra\\
&=\lim_\gamma\la\sigma \h{V}\sigma X\sigma \h{V}^*\sigma\eta\ten\xi_\gamma,\eta\ten\xi_\gamma\ra\\
&=\lim_\gamma\la V^*XV\eta\ten\xi_\gamma,\eta\ten\xi_\gamma\ra\\
&=\la (\id\ten M)V^*XV,\om_\eta\ra.\end{align*}
Since $\eta\in\LT$ was arbitrary, by linearity we obtain
$$\Phi_M(X)=(\id\ten M)(V^*XV), \ \ \ X\in VN(G\times G),$$
from which it follows that $\Phi_M\circ\Gam=\id_{VN(G)}$. The $A(G)$-module property can be deduced from the proof of \cite[Theorem 5.5]{CN}, but we  provide the details for the convenience of the reader.

For $X\in VN(G\times G)$ and $u\in A(G)$, we have
\begin{align*}\Phi_M(u\cdot X)&=\Phi_M((\id\ten\id\ten u)(V_{23}X_{12}V_{23}^*))\\
&=(\id\ten M)(\id\ten\id\ten u)(V_{12}^*V_{23}X_{12}V_{23}^*V_{12})\\
                              &=(\id\ten M)(\id\ten\id\ten u)(V_{13}V_{23}V_{12}^*X_{12}V_{12}V_{23}^*V_{13}^*) \ \ \ \  \textnormal{(by equation (\ref{e:pentagonal}))}\\
                              &=(\id\ten u)(V(\id\ten M\ten\id)(V_{23}V_{12}^*X_{12}V_{12}V_{23}^*)V^*). \end{align*}
Denoting by $\pi:\TC\twoheadrightarrow  A(G)$ the canonical restriction map, and recalling that $M|_{VN(G)}$ is $A(G)$-invariant, for $\tau,\om\in\TC$, we have

\begin{align*}\la(\id\ten M\ten\id)(V_{23}V_{12}^*X_{12}V_{12}V_{23}^*),\tau\ten\om\ra&=\la(M\ten\id)V((\tau\ten\id)(V^*XV)\ten1)V^*,\om\ra\\
                                                                          &=\la M,\pi(\om)\cdot((\tau\ten\id)V^*XV)\ra\\
                                                                          &=\la\om,1\ra\la M,((\tau\ten\id)V^*XV)\ra\\
                                                                          &=\la M\ten\om,(\tau\ten\id)(V^*XV)\ten 1\ra\\
                                                                          &=\la(\id\ten M\ten\id)(V^*XV\ten 1),\tau\ten\om\ra\\
                                                                          &=\la\Phi_M(X)\ten 1,\tau\ten\om\ra.
                                                                          \end{align*}
Since $\tau$ and $\om$ in $\TC$ were arbitrary, it follows that
\begin{align*}\Phi_M(u \cdot X)&=(\id\ten u)(V(\id\ten M\ten\id)(V_{23}V_{12}^*A_{12}V_{12}V_{23}^*)V^*)\\
                              &=(\id\ten u)(V(\Phi_M(X)\ten 1)V^*)\\
                              &= u\cdot \Phi_M(X).\end{align*}
Our original claim is therefore established, and $\Theta(\vphi_\gamma)$ converges weak* in $\mc{CB}(VN(G\times G))$ to an $A(G\times G)$-module projection $\Theta$ from $VN(G\times G)$ onto $VN(G_\Delta)=\Gam(VN(G))$. Then
$$\vphi_{\gamma}|_{G_\Delta}\cdot u-u\rightarrow 0$$
weakly for $u\in A(G_\Delta)$, and using the fact that $\Gam(VN(G))=\{X\in VN(G\times G)\mid (\Gam\ten\id)(X)=(\id\ten\Gam)(X)\}$ \cite[Theorem 6.5]{D4}, together with the essentiality $I(G_\Delta)=\la I(G_\Delta)\cdot A(G\times G)\ra$, we also have
$$\vphi_{\gamma}\cdot v\rightarrow 0$$
weakly for $v\in I(G_\Delta)$. Passing to convex combinations, and noting that $(\vphi_\gamma)\subseteq\mc{P}_1(G\times G)$, we obtain a contractive approximate indicator for $G_\Delta$ in $\mc{P}_1(G\times G)$.
\qed\end{pf}

We conjecture that $A(G)$ is relatively $1$-biflat if and only if $G$ is QSIN, meaning $\LO$ has a bounded approximate identity $(f_\alpha)$ satisfying
$$\norm{\beta_1(s)f_\alpha - f_\alpha}_{\LO}\rightarrow0, \ \ \ s\in G.$$
By Theorem \ref{t:IA=relinj}    and \cite[Corollary 3.2]{LP1}, for any locally compact group $G$ such that $VN(G)$ is 1-injective in $\C-\mathbf{mod}$, relative 1-biflatness of $A(G)$ implies that $G$ is amenable, and therefore QSIN by \cite[Theorem 3]{LR}. Hence, the conjecture is valid for all $G$ such that $VN(G)$ is an injective von Neumann algebra, in particular, for any type I or almost connected group (cf. \cite{Pat3}). We now establish the conjecture for totally disconnected groups.

\begin{prop}\label{p:totallydisconn} Let $G$ be a totally disconnected locally compact group. Then $A(G)$ is relatively 1-biflat if and only if $G$ is QSIN.\end{prop}

\begin{pf} Sufficiency follows from \cite[Theorem 2.4]{ARS}, so suppose that $A(G)$ is relatively 1-biflat. Proceeding as in the proof of Theorem \ref{t:biflat}, we obtain a net of states $(u_\gamma)$ in $A(G)$ satisfying
$$\norm{v\cdot u_\gamma-v(e)u_\gamma}_{A(G)}, \ \norm{s\lhd u_\gamma-u_\gamma}_{A(G)}\rightarrow0$$
for all $v\in A(G)$ and for all $s\in G$.

Now, let $\mathcal{H}$ be a neighbourhood basis of the identity consisting of compact open subgroups. By \cite[Lemme 4.13]{E} for each $H\in\mc{H}$ there exists a state $\vphi_H\in A(G)$ satisfying $\mathrm{supp}(\vphi_H)\subseteq H^2\subseteq H$ and
$$\norm{\vphi_H\cdot v-v(e)\vphi_H}_{A(G)}\rightarrow 0, \ \ \ v\in A(G).$$
For each $H\in\mc{H}$, a standard $3\varepsilon$-argument shows
$$\norm{s\lhd(\vphi_H\cdot u_\gamma)-\vphi_H\cdot u_\gamma}_{A(G)}\rightarrow0, \ \ \ s\in G.$$
Denoting the index set of $(u_\gamma)$ by $\mc{C}$, we form the product $\mc{I}:=\mc{H}\times\mc{C}^{\mc{H}}$. For each $\alpha=(H,(\gamma_H)_{H\in\mc{H}})\in\mc{I}$, letting $u_\alpha:=\vphi_H\cdot u_{\gamma(H)}$, we obtain a net of states in $A(G)$ satisfying the iterated convergence
$$\lim_{\alpha\in\mc{I}}\norm{s\lhd u_\alpha- u_\alpha}_{A(G)}=\lim_{H\in\mc{H}}\lim_{\gamma\in\mc{C}}\norm{s\lhd\vphi_H\cdot u_{\gamma}-\vphi_H\cdot u_{\gamma}}_{A(G)}=0$$
for all $s\in G$ by \cite[pg. 69]{Kelley}. Moreover, $\mathrm{supp}(u_\alpha)\rightarrow\{e\}$, in the sense that for every neighbourhood $U$ of the identity, there exists $\alpha_U$ such that $\mathrm{supp}(u_\alpha)\subseteq U$ for $\alpha\geq\alpha_U$.

Let ($\xi_\alpha$) be the unique representing vectors from $\mc{P}$ for the net $(u_\alpha)$. For each $\alpha=(H,(\gamma_H)_{H\in\mc{H}})\in\mc{I}$, $u_\alpha$ is supported in the open subgroup $H$, i.e., $u_\alpha\in A(H)\subseteq A(G)$. Under the canonical subspace inclusion $L^2(H)\hookrightarrow L^2(G)$ we have $\mc{P}_H=\overline{\{f\ast Jf\mid f\in C_c(H)\}}\subseteq \mc{P}_G$, so by uniqueness of representing vectors \cite[Lemma 2.10]{H}, we may assume $\mathrm{supp}(\xi_\alpha)\subseteq H$.

Applying Haagerup's Powers--St\o rmer inequality \cite[Lemma 2.10]{H} once again, we obtain
\begin{equation*}\norm{\om_{\beta_2(s)\xi_\alpha}-\om_{\xi_\alpha}}^2_{\LO}\leq 4\norm{\beta_2(s)\xi_\alpha-\xi_\alpha}^2_{\LT}\leq 4\norm{s\lhd u_\alpha-u_\alpha}_{A(G)}\rightarrow0, \ \ \ s\in G.\end{equation*}
Letting $f_\alpha:=|\xi_\alpha|^2$, we obtain a net of states in $\LO$ satisfying
$$\norm{\beta_1(s)f_\alpha - f_\alpha}_{\LO}=\norm{\om_{\beta_2(s)\xi_\alpha}-\om_{\xi_\alpha}}_{\LO}\rightarrow0, \ \ \ s\in G,$$
and $\mathrm{supp}(f_\alpha)\rightarrow\{e\}$. Hence, $G$ is QSIN.
\qed\end{pf}

For the semidirect product of an infinite compact group $K$ by a discrete group
$H$, we now show that relative $1$-biflatness of $A(K\rtimes H)$
entails that the unitary representation \[
\pi_{K}:H\rightarrow\mathcal{B}(L_0^2(K)):h\mapsto[\xi\mapsto h\cdot\xi]\]
weakly contains the trivial representation. Here, $L_0^2(K)=\{ \xi\in L^2(K):\int_{K}\xi=0\} $
and $h\cdot\xi(k)=\xi(h^{-1}kh)$
for $h\in H$ and $\xi\in L^2(K)$, where $h^{-1}kh$
is the product in $K\rtimes H$, i.e. the action of $h^{-1}$ on $k$.
If, moreover, the action of $H$ on $K$ is ergodic, we show that
the Haar integral on $K$ is not the unique $H$-invariant mean on
$L^{\infty}(K)$. \emph{Ergodicity} of the $H$-action
on $K$ is the assertion that if $E\subseteq K$ is Borel with $E\triangle h\cdot E$
null for all $h\in H$, then $E$ must be null or co-null, and is equivalent
to the non-existence of normal $H$-invariant means on $L^{\infty}(K)$
other than $1_{K}$.
\begin{prop}\label{p:semidirect}
Let $K\rtimes H$ be the semidirect product of an infinite compact group $K$
by a discrete group $H$. If $A(K\rtimes H)$ is relatively
$1$-biflat, then $\pi_{K}$ weakly contains the trivial representation.\end{prop}
\begin{pf}
Let $G$ denote $K\rtimes H$. As in the proof of Theorem \ref{t:biflat},
relative $1$-biflatness of $A(G)$ yields a net of states
$(\omega_{\xi_{\alpha}})$ in $A(G)$ with $\xi_{\alpha}\in\mathcal{P}_{G}$
satisfying\[
\Vert v\cdot\omega_{\xi_{\alpha}}-v(e)\omega_{\xi_{\alpha}}\Vert _{A(G)}, \ \Vert s\vartriangleleft\omega_{\xi_{\alpha}}-\omega_{\xi_{\alpha}}\Vert _{A(G)}\rightarrow0, \ \ \ v\in A(G), \ s\in G.\]
Arguing as in the proof of Proposition \ref{p:totallydisconn}, we may assume $\mbox{supp}(\omega_{\xi_{\alpha}})\rightarrow\{ e\} $
and, since $K$ is an open subgroup of $G$, we may identify $A(K)$
with a subspace of $A(G)$ and further assume that $\mbox{supp}(\xi_{\alpha})\subseteq K$.
Viewing $L^2(K)$ as a subspace of
$L^{2}(G)$ via extension by zero, we have $\beta_{2}^{G}(h)\xi=h\cdot\xi$ for $\xi\in L^2(K)$ and $h\in H$
by unimodularity of $G$, and, noting once again that $\beta_{2}^{G}(G)\mathcal{P}_{G}\subseteq\mathcal{P}_{G}$,
\cite[Lemma 2.10]{H} implies\[
\Vert h\cdot\xi_{\alpha}-\xi_{\alpha}\Vert _{L^2(K)}^{2}=\norm{ \beta_{2}^{G}(h)\xi_{\alpha}-\xi_{\alpha}} _{L^{2}(G)}^{2}\leq\norm{ \omega_{\beta_{2}^{G}(h)\xi_{\alpha}}-\omega_{\xi_{\alpha}}} _{A(G)}=\Vert h\vartriangleleft\omega_{\xi_{\alpha}}-\omega_{\xi_{\alpha}}\Vert _{A(G)}\rightarrow0\]
for all $h\in H$. Let $\xi_{\alpha}=\xi_{\alpha}^{0}+c_{\alpha}1_{K}$
correspond to the decomposition $L^2(K)=L_0^2(K)\oplus_{2}\mathbb{C}1_{K}$,
so that $1=\Vert \xi_{\alpha}^{0}\Vert _{L_0^{2}(K)}^{2}+|c_{\alpha}|^{2}$
and $\Vert h\cdot\xi_{\alpha}^{0}-\xi_{\alpha}^{0}\Vert _{L_0^{2}(K)}\rightarrow0$ for all $h\in H$.
Fix a neighbourhood $U$ of the identity in $K$ with $|K\setminus U|>0$.
If it were the case that $\Vert \xi_{\alpha}^{0}\Vert _{L_0^{2}(K)}\rightarrow0$,
then, for $\alpha$ large enough that $|c_{\alpha}|^{2}>\frac{1}{2}$
and $\mbox{supp}(\omega_{\xi_{\alpha}})\subseteq U$, we
have for $k\in K\setminus U$ that \[
0=\omega_{\xi_{\alpha}}(k)=\omega_{\xi_{\alpha}^{0}}(k)+\omega_{\xi_{\alpha}^{0},c_{\alpha}1_{K}}(k)+\omega_{c_{\alpha}1_{K},\xi_{\alpha}^{0}}(k)+|c_{\alpha}|^{2}=\omega_{\xi_{\alpha}^{0}}(k)+|c_{\alpha}|^{2}\]
because $\xi_{\alpha}^{0}\in L_0^2(K)$, whence\[
-\frac{1}{2}|K\setminus U|>\int_{K\setminus U}\omega_{\xi_{\alpha}^{0}}(k)dk=\langle \omega_{\xi_{\alpha}^{0}},\lambda_{K}(1_{K\setminus U})\rangle _{A(K),VN(K)}\rightarrow0,\]
a contradiction. Therefore, passing to a subnet if necessary, we may
assume $\Vert \xi_{\alpha}^{0}\Vert _{L_0^{2}(K)}$ is bounded away
from zero, in which case the vectors $\xi_{\alpha}^{0}$ may be normalized
while retaining the property that $\Vert h\cdot\xi^0_{\alpha}-\xi^0_{\alpha}\Vert _{L_0^{2}(K)}\rightarrow0$
for all $h\in H$. Thus $\pi_{K}$ weakly contains the trivial representation.\qed\end{pf}


A locally compact group $G$ is said to have \emph{Kazhdan's
property (T)} if whenever a strongly continuous unitary representation
of $G$ weakly contains the trivial representation it must contain
the trivial representation.

\begin{cor}
Let $K\rtimes H$ be the semidirect product of an infinite compact group $K$
by a discrete group $H$ such that the action of $H$ on $K$ is ergodic.
If $A(K\rtimes H)$ is relatively $1$-biflat, then $H$
does not have Kazhdan's property (T).\end{cor}
\begin{pf}
If $H$ had Kazhdan's property (T), then $\pi_{K}$
would contain the trivial representation and we would obtain a nonzero
vector $\xi\in L_0^2(K)$ such that $h\cdot\xi=\xi$
for all $h\in H$, contradicting the ergodicity of the $H$-action
on $K$.
\qed\end{pf}

This shows, for example, that if $K$ is an infinite
compact group with an ergodic action of $SL(n,\mathbb{Z})$
by automorphisms and $n\geq3$, then the Fourier algebra of $K\rtimes SL(n,\mathbb{Z})$
is not relatively 1-biflat.


The QSIN condition on a locally compact group $G$ is equivalent to
the existence of a conjugation invariant mean on $L^{\infty}(G)$
extending evaluation at the identity on $C_{0}(G)$. In
\cite{LR} it is established that for $n\geq2$ the group $\mathbb{T}^{n}\rtimes SL(n,\mathbb{Z})$
fails to be QSIN by appealing to the fact that the Haar integral on
$\mathbb{T}^{n}$ is the unique mean on $L^{\infty}(\mathbb{T}^{n})$
that is invariant under the $SL(n,\mathbb{Z})$-action.
Indeed, the restriction to $L^{\infty}(\mathbb{T}^{n})$
of any conjugation invariant mean on $L^{\infty}(\mathbb{T}^{n}\rtimes SL(n,\mathbb{Z}))$
is clearly invariant under the action of $SL(n,\mathbb{Z})$.
For semidirect products associated to ergodic actions as above, we
have the following.

\begin{cor}
Let $K\rtimes H$ be the semidirect product of an infinite compact group $K$
by a discrete group $H$ such that the action of $H$ on $K$ is ergodic.
If $A(K\rtimes H)$ is relatively $1$-biflat, then there
is an $H$-invariant mean on $L^{\infty}(K)$ distinct
from the Haar integral on $K$.\end{cor}
\begin{pf}
By \cite[Theorem 1.6]{FS}, $L^{\infty}(K)$ admits an
$H$-invariant mean distinct from the Haar measure when $\pi_{K}$,
considered as a representation on $L_{0}^{2}(K,\mathbb{R})$,
weakly contains the trivial representation. We may assure that the
almost invariant vectors for $\pi_{K}$ produced in Proposition \ref{p:semidirect}
 are real valued by replacing the states $\omega_{\xi_{\alpha}}$ with
$\omega_{\xi_{\alpha}}\overline{\omega_{\xi_{\alpha}}}$, in which
case we have $\omega_{\xi_{\alpha}}\overline{\omega_{\xi_{\alpha}}}=\omega_{\xi_{\alpha}^{\prime}}$
for $\xi_{\alpha}^{\prime}\in\mathcal{P}_{G}$ that are then real-valued
by uniqueness.
\qed\end{pf}

Since the $SL(2,\mathbb{Z})$-action on $\mathbb{T}^{2}$
is ergodic, this confirms that $A(\mathbb{T}^{2}\rtimes SL(2,\mathbb{Z}))$
fails to be relatively 1-biflat. Note, however, that $\mathbb{T}^{2}\rtimes SL(2,\mathbb{Z})$ is an IN group, and hence $A(\mathbb{T}^{2}\rtimes SL(2,\mathbb{Z}))$ is relatively 1-flat by Theorem \ref{t:IA=relinj}. More examples of groups $H$ and $K$ and conditions on these pairs for which there
is a unique $H$-invariant mean on $L^{\infty}(K)$ may be found in \cite{Bekka2} and \cite{FS}.

\section{Operator amenability of $A_{cb}(G)$}

For a locally compact group $G$, let $M_{cb}A(G)$ denote
the completely bounded multiplier algebra of $A(G)$ and
$A_{cb}(G)$ the norm closure of $A(G)$ in
$M_{cb}A(G)$. Given a closed subgroup $H$ of $G$, we
may consider approximate indicators for $H$ consisting of completely
bounded multipliers by replacing $B(G)$ with $M_{cb}A(G)$
in the definition of Section \ref{s:relbiflat}. The existence of an approximate indicator
for $G_{\Delta}$ in the larger algebra $M_{cb}A(G\times G)$
still yields relative biflatness of $A(G)$, the proof
of \cite[Proposition 2.3]{ARS} carrying over mutatis mutandis.

For the algebra $A_{cb}(G)$, the existence
of a bounded approximate identity is equivalent to weak amenability
of $G$ \cite{Fo} and it was suggested in \cite{FRS} that $A_{cb}(G)$
may be operator amenable exactly when $G$ is weakly amenable. The
following proposition, in combination with Corollary \ref{c:F_2},
yields a large class of counter-examples.
\begin{prop}
Let $G$ be a locally compact group such that $A_{cb}(G)$
is operator amenable. Then $G_{\Delta}$ has
a bounded approximate indicator in $A_{cb}(G\times G)$.\end{prop}
\begin{pf}
Write $\Delta:A_{cb}(G)\widehat{\otimes}A_{cb}(G)\rightarrow A_{cb}(G)$
for the product map, $r:A_{cb}(G\times G)\rightarrow A_{cb}(G)$
for restriction to the diagonal $G_{\Delta}$ in $G\times G$, and
$\Lambda:A_{cb}(G)\widehat{\otimes}A_{cb}(G)\rightarrow A_{cb}(G\times G)$
for the complete contraction defined on elementary tensors by $\Lambda(u\otimes v)=u\times v$,
so that $\Delta=r\Lambda$. Let $(X_{\alpha})$ be an approximate
diagonal for $A_{cb}(G)$ of bound $C$ and set $m_{\alpha}=\Lambda(X_{\alpha})$.
We show that the net $(m_{\alpha})$ is an approximate
indicator for $G_{\Delta}$. Let $u\in A(G)$ have compact
support and choose $v\in A(G)$ with $v\equiv1$ on $\mbox{supp}(u)$
\cite[Lemme 3.2]{E}, so that $u=uv$ and \[
\Vert ur(m_{\alpha})-u\Vert _{A(G)}=\Vert u\Delta(X_{\alpha})-u\Vert _{A(G)}\leq\Vert u\Vert _{A(G)}\Vert v\Delta(X_{\alpha})-v\Vert _{A_{cb}(G)}\rightarrow0.\]
As $A(G)$ is Tauberian and the net $(r(m_{\alpha}))$
is bounded in $\Vert \cdot\Vert _{A_{cb}(G)}$,
a routine estimate shows that the above holds for all $u\in A(G)$.

We claim that the elements of $I(G_{\Delta})$ of the form
$(a\times1_{G}-1_{G}\times a)v$ for $a\in A(G)$
and $v\in A(G\times G)$ have dense span. Recall that $A(G)$
is self-induced \cite{D4}, in particular \[
\ker\Delta_{A(G)}=\langle ab\otimes c-a\otimes bc:a,b,c\in A(G)\rangle ,\]
and that the map $a\otimes b\mapsto a\times b$ induces a completely
isometric isomorphism $A(G)\widehat{\otimes}A(G)\rightarrow A(G\times G)$
taking $\ker\Delta_{A(G)}$ onto $I(G_{\Delta})$,
from which it follows that \[
I(G_{\Delta})=\langle ab\times c-a\times bc:a,b,c\in A(G)\rangle .\]
Since $\{ a\times c:a,c\in A(G)\} $ has dense
span in $A(G\times G)$, \begin{eqnarray*}
I(G_{\Delta}) & = & \langle b\cdot(a\times c)-(a\times c)\cdot b:a,b,c\in A(G)\rangle \\
 & = & \langle b\cdot v-v\cdot b:b\in A(G)\mbox{ and }v\in A(G\times G)\rangle \\
 & = & \langle (b\times1_{G}-1_{G}\times b)v:b\in A(G)\mbox{ and }v\in A(G\times G)\rangle .\end{eqnarray*}
For such elements of $I(G_{\Delta})$,\begin{eqnarray*}
\Vert (b\times1_{G}-1_{G}\times b)vm_{\alpha}\Vert _{A(G\times G)} & \leq & \Vert v\Vert _{A(G\times G)}\Vert b\cdot m_{\alpha}-m_{\alpha}\cdot b\Vert _{A_{cb}(G\times G)}\\
 & \leq & \Vert v\Vert _{A(G\times G)}\Vert b\cdot X_{\alpha}-X_{\alpha}\cdot b\Vert _{A_{cb}(G)\widehat{\otimes}A_{cb}(G)}\rightarrow0,\end{eqnarray*}
where the second inequality uses that $\Lambda$ is a contractive $A(G)$-bimodule
map. The density claim above and the boundedness of $(m_{\alpha})$ imply that $\Vert um_{\alpha}\Vert _{A(G\times G)}\rightarrow0$
for all $u\in I(G_{\Delta})$.\qed\end{pf}
\begin{cor}\label{c:Acbopamen}
Let $G$ be a locally compact group containing $\mathbb{F}_{2}$ as
a closed subgroup and for which $VN(G)$ is $1$-injective
in $\mathbb{C}-\mathbf{mod}$. Then $A_{cb}(G)$ is not
operator amenable.\end{cor}
\begin{pf}
If $A_{cb}(G)$ were operator amenable then an approximate
indicator for $G_{\Delta}$ would exist, implying that $VN(G)$
is relatively $C$-injective in $A(G)-\mathbf{mod}$ for some $C\geq 1$ by
the completely bounded multiplier analogue of \cite[Proposition 2.3]{ARS}, in contradiction
to Corollary \ref{c:F_2}.
\qed\end{pf}

Any weakly amenable, non-amenable, almost connected group $G$ satisfies
the hypotheses of Corollary \ref{c:Acbopamen} by \cite{Pat3} and \cite[Theorem 5.5]{Rick}. For example, $SL(2,\R)$, $SL(2,\C)$, and $SO(1,n)$, $n\geq 2$. Since weak amenability is preserved under compact extensions \cite[Proposition 1.3]{CH}
and almost connected groups have injective group
von Neumann algebras, if $K$ is any compact group
with an action of $G$ by automorphisms, then $K\rtimes G$ is weakly
amenable and $A_{cb}(K\rtimes G)$ fails to be operator
amenable.





\section*{Acknowledgements}

This work contains results from the doctoral thesis of the first author, who would like to thank Matthias Neufang for helpful discussions, and was partially supported by an NSERC Canada Graduate Scholarship. We also thank Yemon Choi for helpful comments which improved the presentation of the paper. A portion of this project was completed at the Fields Institute during the Thematic Program on Abstract Harmonic Analysis, Banach and Operator Algebras in 2014, as well as the retrospective meeting in 2015. We are grateful to the Institute for its kind hospitality.

\end{document}